\documentclass[12pt]{article}

\usepackage[dvipsnames]{xcolor}
\usepackage{pdfpages}
\usepackage{amsmath}
\usepackage{amsfonts}
\usepackage{upgreek}
\usepackage{sgame}
\usepackage{accents}
\usepackage{amssymb}
\usepackage{tikz-cd}
\usepackage{float}
\usepackage[numbers]{natbib}   % omit 'round' option if you prefer square brackets
\usepackage{dutchcal}
\usepackage{pdfpages}
\usepackage{amsmath}
\usepackage{amsfonts}
\usepackage{upgreek}
\usepackage{sgame}
\usepackage{mathtools}
\usepackage{amsthm}
\usepackage{amssymb}
\usepackage{enumitem}
\usepackage{epigraph}
\usepackage{sectsty}
\sectionfont{\color{MidnightBlue}}
\subsectionfont{\color{MidnightBlue}}
\subsubsectionfont{\color{MidnightBlue}}

\newtheorem{theorem}{Theorem}[section]
\newtheorem{proposition}[theorem]{Proposition}
\newtheorem{lemma}[theorem]{Lemma}
\newtheorem{corollary}[theorem]{Corollary}
\newtheorem{claim}[theorem]{Claim}

\theoremstyle{definition}

\newtheorem{definition}[theorem]{Definition}
\newtheorem{example}[theorem]{Example}

\newtheorem{question}[theorem]{Question}

\newtheorem{remark}[theorem]{Remark}

\newtheorem{aside}[theorem]{Aside}
\newtheorem{note}[theorem]{Note}

\usepackage{hyperref}
\hypersetup{
    colorlinks=true,
    linkcolor=Maroon,
    filecolor=magenta,      
    urlcolor=Mulberry,
    citecolor=MidnightBlue,
}

\providecommand{\keywords}[1]{\textbf{\textit{Keywords:}} #1}
\providecommand{\jel}[1]{\textbf{\textit{JEL Classifications:}} #1}

\begin{document}

\author{Artem Hulko\thanks{Department of Mathematics and Statistics, University of North Carolina at Charlotte} \and Mark Whitmeyer\thanks{Department of Economics, University of Texas at Austin \newline Email: \href{mailto:mark.whitmeyer@utexas.edu}{mark.whitmeyer@utexas.edu}. \newline This paper has benefited immensely from comments and suggestions by David Angell, V. Bhaskar, William Fuchs, Rosemary Hopcroft, Vasudha Jain, Vasiliki Skreta, Isaac Sonin, Yiman Sun, Joseph Whitmeyer, and Thomas Wiseman. All remaining errors are, naturally, our own.} }

\title{\color{MidnightBlue} A Game of Random Variables}

\date{\today{}}

\maketitle

\begin{abstract}

This paper analyzes a simple game with $n$ players. We fix a mean, $\mu$, in the interval $[0, 1]$ and let each player choose any random variable distributed on that interval with the given mean. The winner of the zero-sum game is the player whose random variable has the highest realization. We show that the position of the mean within the interval is paramount. Remarkably, if the given mean is above a crucial threshold then the unique equilibrium must contain a point mass on $1$. The cutoff is strictly decreasing in the number of players, $n$; and for fixed $\mu$, as the number of players is increased, each player places more weight on $1$ at equilibrium. We characterize the equilibrium as the number of players goes to infinity.
\end{abstract}
\\
\jel{C72; D82; D83}

\newpage

\section{Introduction}

Consider the following problem faced by a group of $n$ manufacturers of some good. Each manufacturer produces the same undifferentiated product, which sells for some exogenously given price. The manufacturers are constrained by the production process into producing goods of the same given average quality; however, they can choose the distribution of the good's quality--by being, say, more by-the-book and rigid a producer can ensure a more constant quality; or by being more flexible and hands-off he can achieve a wider spread of quality realizations. There is some buyer who wants to buy the good, and she naturally would like the good with the highest quality. Moreover, before she makes her choice of product, she may inspect the goods in order to accurately pick out the best one. What distribution over qualities should a producer choose in order to maximize the change that his product is best? 

One might suspect, \textit{ex ante}, that optimal choice of distribution would be the distribution with the highest variance. However, as we show in this paper, that is not the case, and instead it is the uniform distribution that is king.

This model can be formalized as the following $n$-player game. We fix a mean, $\mu$ and have each player simultaneously choose the distribution of a random variable with realizations constrained to a common interval--without loss, $[0,1]$--such that the expectation of the random variable is the given $\mu$. The winner of this zero-sum game is the player with the highest realization of his random variable. Each player's objective; therefore, is to maximize the probability that the realized value of his random variable is higher than that of his opponents. %Without loss, we stipulate that the random variable simply be the identity random variable. Accordingly, each player's strategy consists simply of a choice of distribution supported on a subset of $[0,1]$ such that the expectation of the random variable under that distribution is the given $\mu$. 

Our results are as follows. There is a unique symmetric equilibrium: if $\mu = 1/n$ then both players play a distribution with cdf $F(x) = x^{1/(n-1)}$ supported on the whole interval $[0,1]$, and if $\mu < 1/n$ they each play a distribution with the same curvature on a smaller interval of support, $[0,n\mu]$. The key is that for $j \neq i$, the distribution of $\max_{j \neq i} X_{j}$ is the uniform distribution. On the other hand, if the given mean is greater than $1/n$ then each player places a point mass on $1$ and the remainder of the distribution is continuous, supported on a subset of $[0,1]$. Holding $\mu$ fixed, as the number of players increases, the weight placed on $1$ increases. As $n$ goes to infinity, the unique symmetric equilibrium converges to one where each player chooses a distribution consisting of two point masses on $0$ and $1$. Finally, we show that these results can be extended to a modified case where the maximal support of the distribution is any interval that is a subset of $[0,+ \infty)$.

\subsection{Applications and Discussion}

We believe that this model has numerous applications. Perhaps foremost among these uses is that of competitive information design, or persuasion. This setup models the situation where a group of principals compete over information provision to a risk neutral agent where the principals and the agent share a common prior with binary support. Indeed, our problem is equivalent to one where agents each choose an experiment in the Kamenica and Gentzkow \cite{kam} sense, given the common binary prior. Real world examples of this problem include sellers choosing how much information to convey about their products to a buyer with unknown tastes, schools competitively choosing grading policies for their respective candidates (as in \cite{bol}), and political parties attempting to persuade a voter to choose their candidate (as in \cite{Albrecht}).

More generally, when the probability space is one dimensional this problem of choosing a signal structure becomes one of choosing a Blackwell experiment (see \cite{blackwell}).\footnote{This approach was introduced as a method for examining persuasion problems in Gentzkow and Kamenica \cite{gent}, and a number of other papers have followed that interpret the information design problem as one of choosing an optimal Blackwell experiment. Some other recent examples include \cite{kol} and \cite{skreta}.} As a result, this paper can be interpreted as solving the same problem but with fewer constraints. Put a different way, while the general competitive persuasion question looks for the equilibrium where each player chooses a distribution of posteriors subject to the constraint that the distribution of posteriors be Bayes-plausible, this paper looks for the equilibrium where the constraints are now the two conditions that the maximal support and mean of the posterior and the prior must be the same.

For a prior distribution with binary support, our formulation is precisely the competitive persuasion problem, without any loss of generality. However, it is easy to see that this is not generally true for other prior distributions. For instance, suppose that the prior consisted merely of a weight of $1$ on $1/2$. Obviously, there is no signal structure that could convince a decision maker that the expected value of the object were anything other than $1/2$, and certainly we could not achieve a distribution on posteriors that yielded a uniform distribution on expected values. This intuition is apposite for any distribution with support on $n \neq 2$ points, where $n$ is allowed to be uncountably infinite (i.e. where the prior is non-atomic)--if there is too much weight towards the barycenter of the prior, a uniform distribution of expected rewards is simply not Bayes-plausible.

In part, this paper bears resemblance to Condorelli and Szentes \cite{Condor}. In \cite{Condor}, the authors characterize the equilibrium of a simple game consisting of a monopolistic seller and a buyer. The buyer may choose the cdf of her valuation; and then the seller, after observing the distribution but not the realization, makes an offer to the buyer. This of course is related to Roesler and Szentes \cite{Roes}, where instead of observing her valuation, the buyer merely receives a signal about it. The authors then characterize the properties of this buyer-optimal signal structure. Concordantly, here, we look at the (relatively) unconstrained problem, which provides insight as to the solution to the general competitive persuasion problem.

The approach taken in this paper differs greatly from that taken by other papers in the information design literature. Because the distribution chosen by our players is free from the Bayes' plausibility constraints, we need not utilize the usual approaches--where one concavifies the value function  or uses the experiments-as-convex-functions idea \cite{gent, kol, skreta}. Instead, we can solve for the equilibrium of the game directly, in part using insights gleaned from Hulko and Whitmeyer \cite{Hulko}. The unique equilibrium here has the same intuitive property as the unique equilibrium of the two-player dice game in \cite{Hulko}. As shown by Hulko and Whitmeyer, the famous non-transitivity of generalized dice\footnote{Or, ``Efron's Dice"; see \cite{Con, Savage, Tenney}.}, results in cycles of best responses a la ``Rock-Paper-Scissors"; e.g. die $A$ beats die $B$ beats die $C$ beats die $A$. The only die impervious to this is the standard die (the analog of the uniform distribution in this paper), which guarantees a payoff of $1/2$ to both players. Thus, while the standard die does not beat any other die, it does not lose to any other die either.

In this paper, we arrive at a similar result in the two-player case, which is a corollary of our $n$-player result. Indeed, this result for two-players features as a lemma (Lemma 4.1) in Boleslavsky and Cotton (2015) \cite{cotton}, who look at a two-player problem in which schools must choose students to admit to their respective schools and then design a competitive grading policy. For a given mean less than or equal to $1/2$, the unique equilibrium is the uniform distribution, which guarantees a payoff of at least $1/2$ to the player who chooses it. Then, when $\mu > 1/2$, we have to modify this somewhat because the uniform distribution is no longer an equilibrium. Surprisingly, the unique equilibrium must have a point mass at $1$ and then a portion with a linear distribution.

For a general number of players $n$, the crucial cutoff is $1/n$; and moreover, the equilibrium distribution is the $(n-1)$-th root of $x$. Thus, as $n$ grows the continuous portion of the distribution becomes increasingly concave. As in the two-player case, if the mean is too high (greater than $1/n$) then at equilibrium each player puts a point mass on $1$. Then, for a fixed mean, as the number of players increases, each player places increasingly more weight as a point mass on $1$. As $n$ goes to infinity, the equilibrium distribution converges to one consisting merely of point masses on $1$ and $0$, of weight $\mu$ and $1-\mu$, respectively.

Finally, we would be remiss should we not mention the other papers in the realm of competitive information design. The area of Bayesian Persuasion and information design is growing rapidly, but there are still relatively few papers that look at competitive information provision. As mentioned above, this paper generalizes a result of Boleslavsky and Cotton \cite{cotton}, and the methods used in this paper--the use of calculus of variations techniques to directly solve for the equilibrium distribution--are novel in the literature. Moreover, because we derive a full characterization of the equilibrium for a general $n$, we are able to fully characterize the effect of an increase in population size on competitive information provision.\footnote{It has recently come to our attention that Spiegler \cite{spieg}, in looking at firms selling to a boundedly rational consumer, formulates a model that is extremely similar to ours in which the given mean is endogenous. The equilibrium in that study as well as the resulting insights are virtually the same as the results here, although our approach is novel.}

An additional paper similar to this one is that by Koessler, Laclau, and Tomala \cite{Koessler}. Each player in their game designs a signal structure that contains information about their respective (independent) pieces of information, with the goal of persuading the decision-maker to take the player's preferred action. Similarly, in a pair of papers, Au and Kawai \cite{Au1, Au2} also look at the situation where a number of persuaders compete through information provision. Board and Lu \cite{Board}, in turn, look at information in a search setting. In concert with the insight derived in \cite{Au2} and \cite{Board} we establish that competition elicits greater information provision. In another paper, Boleslavsky and Cotton \cite{bol} look at another two-player game of competitive persuasion in which agents provide information in order to secure funding for proposals. Similarly, Albrecht \cite{Albrecht}, looks a two-player game consisting of two parties vying for the support of a voter. His paper is another real world manifestation of the two-player case of our problem.%Just as full disclosure is achieved as the number of players (\cite{Au2}) goes to $\infty$, or as  search costs (\cite{Board}) go to $0$, we see that as the number of players goes to infinity, . Instead, here, increasing the number of players merely results in an equilibrium nonexistence result.

\subsection{Formulating the Problem}

Formally, let the sample space, $\Omega$, be the closed interval $[0,1]$ and $\mathcal{B}$ be the $\sigma$-Algebra of Borel sets on $[0,1]$. Define random variable $X_{i}$, $i \in \big\{1,2,\dots, n\big\}$ as the identity random variable.  We define a strategy for a player $i$ as follows:

\begin{definition}
Fix mean $\mu \in (0,1)$. A \textcolor{Maroon}{Strategy} for player $i$ consists of a choice of probability distribution $F_{i}$ such that $\mathbb{E}_{F_{i}}\big[X_{i}\big] = \mu$. Write a player $i$'s strategy as the duple $S_{i} \coloneqq (F_{i},X_{i})$.
\end{definition}

The game we shall analyze is constant sum and symmetric; and the payoff for Player $1$, $u_{1}(S_{1},S_{-1})$, is given by

\[\begin{split}
    u_{1}(S_{1},S_{-1}) = \Pr(X_{1} > \max_{j \neq 1}(X_{j}) &+ \frac{1}{n}\Pr(X_{1} = X_{2} = \dots = X_{n})\\ 
    &+ \binom{n-1}{n-2}\frac{1}{n-1}\Pr(X_{1} = X_{2} = \dots = X_{n-1} > X_{n})\\
    &\vdots\\
    &+ \binom{n-1}{1}\frac{1}{2}\Pr(X_{1} = X_{2} > \max_{j \neq 1,2}X_{j})
\end{split}\]

In other words, players want their random variable to have the highest realization, and ties are broken fairly. Before continuing we wish to make the following remark,

\begin{remark}
The set of mixed strategies is equal to the set of pure strategies.
\end{remark}

To see this, note that every mixed strategy consists of some randomization over a set of pure strategies, which is a probability distribution over probability distributions of the random variable. However, this itself is clearly a probability distribution of the random variable and so it is a pure strategy.

\section{The \texorpdfstring{$n$}{TEXT}-Player Game}

First, we write the following lemmata. The intuition behind them is straightforward; simply, there are no symmetric equilibria in which players choose discrete distributions. A player may always deviate profitably from such distributions by shaving a small amount of weight from the highest point in the support and moving this weight to the other points in the support.

\begin{lemma}\label{discrete}
There are no symmetric Nash Equilibria where players choose discrete distributions supported on $N(<\infty)$ points. 
\end{lemma}

\begin{proof}

It is easy to see that there is no symmetric equilibrium in which each player chooses a distribution consisting of a single point mass. Such a distribution could only consist of distribution with weight $1$ placed on $\mu$ and would yield to each player a payoff of $1/n$. However, there is a profitable deviation for a player to instead place weight $1 - \epsilon$ on $\mu + \eta$ and weight $\epsilon$ on $0$ ($\epsilon, \eta > 0$). In doing so, this player could achieve a payoff arbitrarily close to $1$.

Now, assume $\infty > N \geq 2$. Observe that a strategy consists of a choice of probabilities $\big\{p_{1}, p_{2}, \dots, p_{N}\big\}$, $p_{i} \in [0,1]$ $\forall i$,\quad $\sum_{i=1}^{N}p_{i} = 1$ \quad and support \quad $a_{1} < a_{2} < \cdots < a_{N} \in [0,1]$ \quad such that $\sum_{i=1}^{N} a_{i}p_{i} = \mu$. %First we show that $a_{N}$ must be $1$. We prove by contradiction. Suppose that there is a Nash Equilibrium where $a_{N} < 1$ and call the equilibrium strategy $E$. 

The expected payoff to each player from playing an arbitrary strategy, $S_1 = S_2=\dots=S_n = E$, is

\[
    u_i(S_i,S_{-i})  = \sum_{j=0}^{N-1}\left(\sum_{i=0}^{n-1} \binom{n-1}{i} \frac{1}{m-i} p_{N-j}^{n-i}\left(\sum_{k=1}^{N-j-1}p_k\right)^i\right) 
\]

We claim deviating to the following strategy is profitable: $S_1' $ where $a_{N}^{'} = a_{N}$ is played with probability $p_{N} - \epsilon$ and $a_{j}^{'} = a_j+\eta$ is played with probability $p_{j} + \epsilon_{j}$, for $j \neq N$, where $\sum_{j}^{N-1}\epsilon_{j} = \epsilon$ (Again, $\epsilon, \eta, \epsilon_{j} > 0$ $\forall j$).\footnote{Note that we can always find such an $\eta > 0$.} The expected payoff to player 1 playing strategy $S_1'$ is

\[\begin{split}
u_{1}(S_1',S_{-1}) &= \sum_{j=0}^{N-1}\left(\sum_{i=0}^{n-1} \binom{n-1}{i} \frac{1}{n-i} p_{N-j}^{n-i}\left(\sum_{k=1}^{N-j}p_k\right)^i\right)\\
&- \epsilon \sum_{i=0}^{n-1} \binom{n-1}{i} \frac{1}{n-i} p_{N}^{n-i-1} \left( 1-p_N \right)^i\\
&+\sum_{j=1}^{N-1}\left(\sum_{i=0}^{n-1} \binom{n-1}{i} \frac{n-i-1}{n-i} p_{N-j}^{n-i}\left(\sum_{k=1}^{N-j}p_k\right)^i\,\right)
\end{split}\]
Note that the deviation is profitable for the player 1 if
\[
\epsilon \sum_{i=0}^{n-1} \binom{n-1}{i} \frac{1}{n-i} p_{N}^{n-i-1} \left( 1-p_N \right)^i < \sum_{j=1}^{N-1}\left(\sum_{i=0}^{n-1} \binom{n-1}{i} \frac{n-i-1}{n-i} p_{N-j}^{n-i}\left(\sum_{k=1}^{N-j}p_k\right)^i\,\right)\,,
\]

Which holds for a sufficiently small vector $(\epsilon_{1},...,\epsilon_{N-1})$.

\end{proof}

Moreover, we can also show that there can be no distributions with point masses on any point in $[0,1)$; to wit,

\begin{lemma}
There are no symmetric Nash equilibria with point masses on any point in the interval $[0,1)$. Moreover if $\mu < 1/n$, then all symmetric equilibria must be atomless.
\end{lemma}

\begin{proof}
Using an analogous argument to that used in Lemma \ref{discrete}, it is easy to see that there cannot be multiple point masses. Accordingly, it remains to show that there cannot be a single point mass. First, we will show that there cannot be an atom at any point $b \in (0,1)$. Suppose for the sake of contradiction that that there is a symmetric equilibrium where each player plays a point mass of size $p$ on point $b$. That is, each player plays strategy $S$ that consists of a distribution $F$ and a point mass of size $p$ on point $b$. Let $H(x) = F^{n-1}$. Then, player $1$'s payoff is

\[u_{1}(S_{1}, S_{-1})= \int_{0}^{1}\int_{0}^{y}h(x)f(y)dxdy + p\sum_{i=0}^{n-1}\binom{n-1}{i}\bigg(\frac{1}{n-i}\bigg)F(b)^{i}p^{n-1-i}\]

Then, let player $1$ deviate by introducing a tiny point mass of size $\epsilon$ at $0$ and moving the other point mass to $b + \eta$ and reducing its size slightly to $p - \epsilon$ ($\epsilon, \eta > 0$); call this strategy $S_{1}^{'}$. The payoff to player $1$ is

\[u_{1}(S_{1}^{'}, S_{-1})= \int_{0}^{1}\int_{0}^{y}h(x)f(y)dxdy + (p-\epsilon)\sum_{i=0}^{n-1}\binom{n-1}{i}F(b + \eta)^{i}p^{n-1-i}\]

Suppose that this is not a profitable deviation. This holds if and only if

\[\begin{split}
    p\sum_{i=0}^{n-1}\binom{n-1}{i}\bigg(\frac{1}{n-i}\bigg)F(b)^{i}p^{n-1-i} &\geq (p-\epsilon)\sum_{i=0}^{n-1}\binom{n-1}{i}F(b + \eta)^{i}p^{n-1-i}\\
\end{split}\]

Or,

\[\begin{split}
    \epsilon p^{n-1} &+ p\sum_{i=1}^{n-1}\binom{n-1}{i}\bigg(\frac{1}{n-i}\bigg)F(b)^{i}p^{n-1-i}\\
    &\geq \frac{n-1}{n}p^{n} + (p-\epsilon)\sum_{i=1}^{n-1}\binom{n-1}{i}F(b + \eta)^{i}p^{n-1-i}
\end{split}\]

Clearly, as $\epsilon$ and $\eta$ go to zero we achieve a contradiction. Hence, there is a profitable deviation and so this is not an equilibrium. It is clear that there cannot be an equilibrium with a point mass on $0$ and so we omit a proof. Finally, we may conclude from the analysis in the sections \textit{infra} that there may not be a point mass on $1$ if $\mu \leq 1/n$.
\end{proof}

As will become clear, the value of $\mu$ is important in determining the equilibrium of this game. We divide our analysis into the following two cases:

\begin{enumerate}[label=\textbf{\arabic*})]
    \item $\mu \geq \frac{1}{n}$ (Section \ref{=}); and
    \item $\mu < \frac{1}{n}$ (Section \ref{<}).
\end{enumerate}

\subsection{\texorpdfstring{$\mu \geq \frac{1}{n}$}{text}}\label{=}

The main result of this section is the following theorem:

\begin{theorem}\label{n>}
In the game with $n$ players, if $\mu \geq 1/n$ then the unique symmetric Nash equilibrium is for each player to play $\mathcal{F}_{i}$, defined as

\[F_{i}(x) = (1-a)\bigg(\frac{x}{s}\bigg)^{1/(n-1)} \hspace{1cm} \text{for} \hspace{.2cm} x \in [0, s]\]

where $a = \mu - \mu(1-a)^{n}$ and $s = n\mu(1-a)^{n-1}$; and $\Pr(X=1) = a$.
\end{theorem}

\begin{proof}

First, we show that this is an equilibrium. Accordingly, we need to show that there can be no unilateral profitable deviation. Define $Z$ as $\max_{i \neq 1}X_{i}$, which, recall has a point mass on $1$. Moreover, define $H$ as the corresponding continuous portion of the distribution of $Z$; $H \coloneqq F_{i}^{n-1}$:

\[H(z) = (1-a)^{n-1}\frac{x}{s}\]

with associated density

\[h(z) = \frac{(1-a)^{n-1}}{s}\]

Evidently, it suffices to show that our candidate strategy achieves a payoff of at least $1/n$ to the player who uses it, irrespective of the strategy choice by the other players. Suppose for the sake of contradiction that there is a profitable deviation, that is, player $1$ deviates profitably by playing strategy $\mathcal{G}$. Clearly, we can represent $\mathcal{G}$ as having a point mass of size $c$, $0 \leq c \leq \mu$ on $1$ (naturally, if $c = 0$, then there is no point mass there). Written out, $\mathcal{G}$ consists of

\[\label{gee2}\tag{$\mathcal{G}$}
      G(y) \hspace{1cm} \text{for} \hspace{.2cm} x \in [0, 1)
\]

and $\Pr(Y=1)= c$. Define $K \coloneqq \int_{0}^{1} dG = 1-c$. Naturally, $K \leq 1$.  Then, player $1$'s utility from this deviation, $u_{1}(\mathcal{G},S_{-1})$, is\footnote{Note that the first term, $c(1-(1-a)^n)/na$, is derived below, in the proof of Lemma \ref{blah}.}

\[\begin{split}
    u_{1} &= c\frac{1 - (1-a)^{n}}{na} + (1-a)^{n-1}\big(K - G(s)\big) + \int_{0}^{s}\int_{0}^{y}h(x)g(y)dxdy\\
    &= \frac{c}{n\mu} + (1-a)^{n-1}\big(K - G(s)\big) + \int_{0}^{s}\int_{0}^{y}h(x)g(y)dxdy
\end{split}\]

Evidently, this is a profitable deviation if and only if $u_{1} > 1/n$; that is,

\[\begin{split}
    \frac{c}{n\mu} + (1-a)^{n-1}\big[K - G(s)\big] + \int_{0}^{s}\frac{(1-a)^{n-1}}{s}yg(y)dy &> \frac{1}{n}\\
\end{split}\]

After some clever manipulation\footnote{See Appendix \ref{a01} for the detailed derivation.}, this reduces to

\[\tag{$1$}\label{101}\begin{split}
    K &> \int_{s}^{1}\frac{1}{s}yg(y)dy + G(s)\\
\end{split}\]

It is clear that $\int_{s}^{1}\frac{1}{s}yg(y)dy \geq \int_{s}^{1}g(y)dy$ and thus we have

\[\begin{split}
    K > \int_{s}^{1}\frac{1}{s}yg(y)dy &+ G(s) \geq \int_{0}^{s}g(y)dy + \int_{s}^{1}g(y)dy = K
\end{split}\]

We have established a contradiction and thus the result is shown.

\end{proof}

It remains to show uniqueness. To do this, we derive the candidate strategy presented in the theorem above. First, through analogous arguments to those presented \textit{supra}, it is clear that there cannot be multiple mass points, nor can there be any mass point in the interval $[0,1)$. Hence, we allow for there to possibly be a point mass on $1$, and show it must satisfy the following inequality.

\begin{lemma}\label{blah}
Suppose that in a symmetric equilibrium each player puts a point mass of size $a \geq 0$ on $1$. Then, $a$ must satisfy $a \geq \mu\big[1 - (1-a)^{n}\big]$.
\end{lemma}
\begin{proof}

Let each player play strategy $S_{i} = S$ where they each put weight $a$ on $1$. Suppose that player $1$ deviates and plays strategy $\hat{S}_{1}$ consisting of random variable $Y$ distributed with value $1$ with probability $\mu$ and $0$ with probability $1-\mu$.

Then, player $1$'s payoff is

\[\tag{$2$}\label{6}u_{1}(\hat{S}_{1}, S_{-1}) = \mu \sum_{i=0}^{n-1}\binom{n-1}{i}\bigg(\frac{1}{n-i}\bigg)(1-a)^{i}a^{n-1-i}\]

Through judicious use of the binomial theorem\footnote{See Appendix \ref{a0}.}, we can write \ref{6} as 

\[\tag{$3$}\label{l}u_{1}(\hat{S}_{1}, S_{-1}) = \mu\frac{1 - (1-a)^{n}}{na}\]

This must be less than or equal to $1/n$:

\[\label{b3}\tag{$4$}\begin{split}
    \frac{1}{n} &\geq \mu\frac{1 - (1-a)^{n}}{na}\\
    a &\geq \mu\big[1 - (1-a)^{n}\big]\\
\end{split}\]

\end{proof}

There must also be a continuous portion of the distribution on some interval $[t,s]$ with $t \geq 0$, $s \leq 1$. Accordingly, our candidate equilibrium strategy, $\mathcal{F}_{i}$, is of the following form

\[\mathcal{F}_{i} = \begin{cases}\label{f}\tag{$\mathcal{F}_{i}$}
      0 & x \in [0,t)\\
      F_{i} & x \in [t,s) \\
      1-a & x = s
   \end{cases}
\]
with $0 \leq t < s \leq 1$ and $\Pr(X=1)=a$.\footnote{It is clear that the weight on $1$, $a$, cannot be $\mu$ in the symmetric equilibrium. To see this, note that such a value for $a$ would beget a distribution with binary support, which we already ruled out in Lemma \ref{discrete}.} 

We look for a symmetric equilibrium. Observe that distributions $F_{i}$ must be such that 

\[\int_{t}^{s}xf_{i}(x)dx = \mu - a\]

Fix $F_{j}$ for $j \neq i$ and define $H$ as $F_{j}^{n-1}$. Given this distribution, we have the necessary condition that $F_{i}$ maximizes
\[\frac{1 - (1-a)^{n}}{n} + \int_{t}^{s}f_{i}(x)H(x)dx \]

We use calculus of variations techniques and so we define the functional $J[f]$ as the Euler-Lagrange equation

\[J[f_{i}] = \int_{t}^{s}f_{i}(x)H(x)dx - \lambda_{0}\bigg[\int_{t}^{s}f_{i}(x)dx - (1-a)\bigg] - \lambda\bigg[\int_{t}^{s}xf_{i}(x)dx - \mu + a\bigg]\]

The first constraint ensures the distribution satisfies Kolmogorov's second axiom, and the second constraint guarantees that the expectation is $\mu$. %We consider any distribution $\phi$ with $f_{i}(0) = f_{i}(1) = \phi(0) = \phi(1)$ and calculate

%\[\begin{split}
%\left.\frac{\delta J}{\delta \phi}\right\vert_{f} &= \lim_{\epsilon \to 0} \frac{J[f + \epsilon \phi] - J[f]}{\epsilon}\\
%&= \lim_{\epsilon \to 0}\frac{1}{\epsilon}\bigg[\int_{0}^{1}\epsilon \phi H(x)dx  - \lambda_{0}\bigg(\int_{0}^{1}\epsilon \phi dx\bigg) - \lambda\bigg(\int_{0}^{1}\epsilon \phi x dx\bigg)\bigg]
%\end{split}\]

The functional derivative is then

\[\begin{split}
\left.\frac{\delta J(f(x))}{\delta f(x)}\right. &= H(x) - \lambda_{0} - \lambda x
\end{split}\]

This must equal $0$ at a maximum, so we have,
\[H(x) = \lambda_{0} + \lambda x\]

Then, by symmetry, $H(\cdot) = F_{i}^{n-1}(\cdot)$. Moreover, we have two initial conditions that allow us to obtain $t$ and $s$. Using the conditions $F_{i}(t) = 0$ and $F_{i}(s) = (1-a)$,  the equilibrium distribution, $F_{i}$, must be

\[F_{i}(x) = (1-a)\bigg(\frac{x-t}{s-t}\bigg)^{1/(n-1)}\]

with the corresponding pdf,

\[f_{i}(x) = \frac{1}{x-t}\bigg(\frac{1-a}{n-1}\bigg)\bigg(\frac{x-t}{s-t}\bigg)^{1/(n-1)}\]

Note that we also need $\int_{t}^{s}xf_{i}(x)dx = \mu - a$, which reduces to

\[\tag{$5$}\label{2}\begin{split}
a = \frac{n\mu - \big[s + (n-1)t\big]}{n - \big[s + (n-1)t\big]}
\end{split}\]

We now show that $t$ must be $0$:

\begin{lemma}\label{lem7}
The lower bound of the continuous portion of the distribution, $t$, must be $0$.
\end{lemma}

\begin{proof}
We leave the detailed proof for Appendix \ref{a1}. Our proof is through contradiction; we suppose that there is a symmetric equilibrium in which $t > 0$, and show that there exists a profitable deviation.
\end{proof}

Finally, we pin down the size of the weight on $1$:

\begin{lemma}\label{lem8}
The weight on $1$, $a$, is given by $a = \mu - \mu(1-a)^{n}$.
\end{lemma}

\begin{proof}
The detailed proof may be found in Appendix \ref{a2}.
\end{proof}
\begin{corollary}\label{core}
If $\mu = 1/n$, then the unique symmetric Nash equilibrium is for each player to play strategy $F_{i} \coloneqq \big(\frac{x}{n\mu}\big)^{1/(n-1)}$ supported on $[0,n\mu]$.
\end{corollary}

\begin{proof}
We need to show that $a = 0$ and $s = 1$. Recall that we have $a = \mu - \mu(1-a)^n$, which becomes $0 = (1-a)^{n} + na - 1$ when $\mu = 1/n$. It is easy to see that this polynomial has a root at $a = 0$. For convenience, define $b = 1-a$, and after rearranging we obtain %\[(1-b) - \mu + \mu b^{n} = 0\] Or,

\[\tag{$6$}\label{4}\begin{split}
    \mu &= \frac{1-b}{1-b^{n}}\\
\end{split}\]

or, for $\mu = 1/n$,

\[\begin{split}
    \frac{1}{n} &= \frac{1-b}{1-b^{n}}\\
\end{split}\]

Define $\varphi \coloneqq (1-b)/(1-b^{n}) - \mu$. Clearly, $\varphi$ is decreasing in $b$ for $b \in [0,1]$, and is therefore increasing in $a$ over the same domain. As a result, $a$ must be $0$. We can substitute this into $s = (1-a)^{n-1}$ and obtain that $s = 1$.
\end{proof}

We write the following result, which describes the effect of an increase in the number of players.

\begin{theorem}\label{lemlem}
Fix $\mu>1/n$. Then, if the number of players is increased, the weight placed on $1$ in the symmetric equilibrium must increase. That is, $a$ is strictly increasing in $n$. Moreover, $s$ is strictly decreasing in $n$.

In the limit, as the number of players, $n$, becomes infinitely large, the weight on $1$ converges to $\mu$. That is, the equilibrium distribution converges to a distribution with support on two points, $1$ and $0$.
\end{theorem}

\begin{proof}
Define $b = 1-a$ and $\varphi$ as above (the right hand side of \ref{4}). Recall that for $b \in [0,1]$, $\varphi$ is decreasing in $b$ and therefore increasing in $a$ over the same interval. Moreover, we make take the partial derivative with respect to $n$:

\[\frac{\partial{\varphi}}{\partial{n}} = \frac{(1-b)b^{n}\ln(b)}{(b^{n}-1)^{2}} < 0\]

Thus, as $n$ increases, the $a$ needed to satisfy the above expression must increase.  That is, more and more weight is put on $1$. Concurrently, $s$, or the upper bound of the continuous portion of the distribution is shrinking, since, recall \[s = \frac{n(\mu-a)}{1-a}\]

and thus \[\frac{\partial{s}}{\partial{a}} = \frac{-n(1-\mu)}{(1-a)^{2}}\]Furthermore, as $n$ goes to infinity, we see that $a$ goes to $\mu$.

\end{proof}

An illustration of the relationship between $n$ and the symmetric equilibrium values of $a$ and $s$ is given in Figure \ref{figgypudding}.

\begin{figure}
\begin{center}
\includegraphics{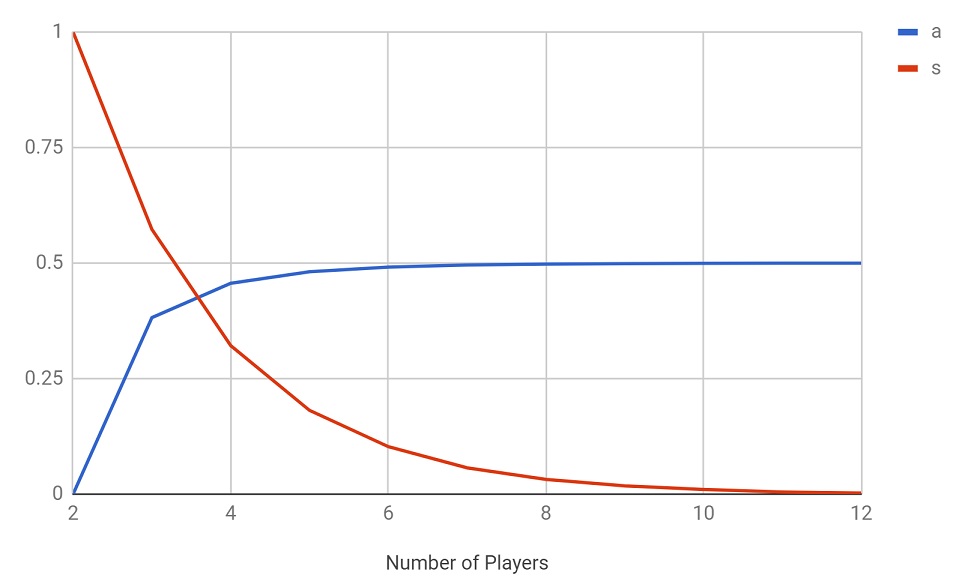}
\end{center}
\caption{\label{figgypudding} Effect of $n$ on $s$ and $a$ at Equilibrium for $\mu = 1/2$}
\end{figure}

\subsection{\texorpdfstring{$\mu < \frac{1}{n}$}{text}}\label{<}

We write, simply

\begin{theorem}\label{n<}
If $\mu < 1/n$ then the unique symmetric Nash equilibrium is for each player to play strategy $F_{i} \coloneqq \big(\frac{x}{n\mu}\big)^{1/(n-1)}$ supported on $[0,n\mu]$.
\end{theorem}

\begin{proof}
From Corollary \ref{core}, it is clear that this distribution is the unique symmetric Nash equilibrium for distributions restricted to $[0,n\mu]$. However, here we wish to show that this distribution is the unique symmetric Nash equilibrium even for deviations on the whole of $[0,1]$.

Consider the case where all $n$ players play strategy $S_{i}$, in which $X_{i}$ has distribution $F_{i}(x) = \big(\frac{x}{n\mu}\big)^{1/(n-1)}$. Suppose a player, say player $1$, deviates, and plays any other strategy $\hat{S}_{1} = G(y)$ supported on $[0,1]$ with mean $\mu$. We wish to show that the probability that $\max_{i \neq 1}X_{i} < Y$ is less than or equal to $1/n$.

For convenience define the new random variable $Z \coloneqq \max_{i \neq 1}X_{i}$. Evidently, $Z$ has distribution $H(z) = \frac{z}{n\mu}$, with associated density $h(z) = \frac{1}{n\mu}$.

Suppose for the sake of contradiction that this deviation is profitable for player $1$: $u_{1}(\hat{S}_{1},S_{-1}) > 1/n$. Then, 

\[\begin{split}
u_{1}(\hat{S}_{1},S_{-1}) &= \big(1 - G(n\mu)\big) + \int_{0}^{n\mu}\int_{0}^{y}g(y)h(z)dzdy\\
&= \big(1 - G(n\mu)\big) + \int_{0}^{n\mu}\frac{y}{n\mu}g(y)dy\\
\end{split}\]
Thus, our supposition above holds if and only if

\[\begin{split}
\big(1 - G(n\mu)\big) + \int_{0}^{n\mu}\frac{y}{n\mu}g(y)dy &> \frac{1}{n}\\
\frac{1}{n} - \int_{n\mu}^{1}\frac{y}{n\mu}g(y)dy &> G(n\mu) - \frac{n-1}{n}\\
\end{split}\]
Or,

\[\begin{split}
1 &> G(n\mu) + \int_{n\mu}^{1}\frac{y}{n\mu}g(y)dy \geq G(n\mu) + \int_{n\mu}^{1}g(y)dy = 1\\
\end{split}\]
We have established a contradiction and thus we conclude that there is no profitable deviation. Uniqueness is immediate, following similar arguments to those used in proving Theorem \ref{n>}.
\end{proof}

\subsection{Any Positive Support}

We finish by solving for the unique equilibrium when players may choose any positive support for the random variable $X_{i}$. We write the following theorem.

\begin{theorem}Fix $\mu > 0$. The unique symmetric Nash equilibrium is for each player to play strategy $F_{i} \coloneqq \big(\frac{x}{n\mu}\big)^{1/(n-1)}$ supported on $[0,n\mu]$.
\end{theorem}

The proof is analogous to that for Theorem \ref{n<} and so is omitted. This problem, of course, is sensible should our interpretation be one of sellers choosing a distribution of qualities for a product. If on the other hand, our problem is one of competitive persuasion (with the Bayes-plausibility constraint relaxed), this environment (allowing for the support chosen to be any subset of $\Re_{+}$) is not appropriate.

\subsection{Arbitrary Ranking}

We can further generalize the model by characterizing the unique symmetric equilibrium for the situation where each player's objective is to have at least the $k$th highest realization (with ties settled randomly). This would correspond to the scenario with $k \geq 1$ buyers.

\section{Brief Discussion}

To put succinctly, in the game examined in this paper, the uniform distribution (or more generally, the distribution $F(x) = x^{1/(n-1)}$) is supreme. What is important is that the distribution of the random variable $\max_{j \neq i}X_{j}$; the distribution each player $i$ faces, is uniform. The intuition behind this result is simple; the uniform does not allow for one's opponent to achieve a payoff higher than $1/n$. If the exogenously given mean is too high; however, any continuous distribution on the interval falls vulnerable to deviation by putting a point mass at $1$ and so to counter this at equilibrium players must put a point mass on $1$. 

As the number of players grows, the cutoff beyond which a point mass is necessary shrinks and the size of the point mass on $1$ grows. Moreover, as the number of players becomes infinitely large, each player will have to put all of the weight $\mu$ on the point $1$. If however, the players can choose any positive support for their distribution then they can continue to enlarge the interval of support as $n$ grows and they never have to include a point mass.%For whatever other vector of strategies we choose, there is a profitable deviation.

Thus, we can see that there is interesting intuition that can be gleaned from this problem. If the mean is relatively small ($\mu < 1/n$), then the players do not use the top portion of their intervals at all, since that portion of the distribution is too ``valuable", so-to-speak, and is better spread out over the lower portion of the interval. If we think of the problem as one of a posterior distribution over prizes induced by signals, this means that there is no fully informative signal for the highest realization of the prize. This runs counter to the seminal result from \cite{kam}, where the highest state always induces the ``high" signal. 

Then, as the number of players grows, players will use more and more of the top portion of the intervals. Furthermore, if the number of players is sufficiently high, or equivalently, if the mean becomes sufficiently large, then the players must put a point mass at $1$. As the number of players grows beyond the cutoff, each player will increase the weight on $1$ and the support of the continuous portion of the distribution will shrink. Finally, in the limit, equilibrium will consist of two point masses. Looking again at the problem as the choice of posteriors, we see that the players become increasingly ``honest" about the high realization. Consonantly, as $n$ grows (in the persuasion setting with a binary prior) we get full revelation in the limit--competition forces the players to reveal everything. %Of course, if there are no constraints on the maximal support of the distribution, each player can always just play the continuous distribution supported on the apposite interval, and they never put a point mass on $1$.

From the point of view of the consumer; then, an increase in the number of players is a good thing. They would ideally draw from distributions with the maximum variance, which is the distribution supported on the end points of the interval. Thus, as the number of players becomes infinitely large, the equilibrium converges to one that is consumer-optimal.

Overall, this paper can seen as an analysis of a competitive persuasion problem for the situation in which the agents have a common binary prior, or for the general competitive persuasion problem where the Bayes-plausibility constraint is relaxed. Moreover, our results apply to a variety of problems beyond information design. Our setup models the unconstrained version of competition between agents who each must mix or choose a mixture of some type or quality, which mixtures are each then randomly sampled from and the highest chosen. 

\bibliography{sample.bib}

\appendix

\section{Appendix}\label{appendix1}

\subsection{Equation \texorpdfstring{\ref{101}}{TEXT} Derivation}\label{a01}

\begin{claim}
\[\frac{c}{n\mu} + (1-a)^{n-1}\big[K - G(s)\big] + \int_{0}^{s}\frac{(1-a)^{n-1}}{s}yg(y)dy > \frac{1}{n}\]
is equivalent to
\[\begin{split}
    K &> \int_{s}^{1}\frac{1}{s}yg(y)dy + G(s)\\
\end{split}\]
\end{claim}

\begin{proof}
\[\begin{split}
    \frac{c}{n\mu} + (1-a)^{n-1}\big[K - G(s)\big] + \int_{0}^{s}\frac{(1-a)^{n-1}}{s}yg(y)dy &> \frac{1}{n}\\
    \frac{c-\mu}{n\mu(1-a)^{n-1}} + K + \int_{0}^{s}\frac{1}{s}yg(y)dy &> G(s)\\
    \frac{c-\mu}{n\mu(1-a)^{n-1}} + K + \int_{0}^{1}\frac{1}{s}yg(y)dy &> G(s) + \int_{s}^{1}\frac{1}{s}yg(y)dy\\
    \frac{c-\mu}{n\mu(1-a)^{n-1}} + \frac{\mu - c}{s} + K  &> G(s) + \int_{s}^{1}\frac{1}{s}yg(y)dy\\
    K &> \int_{s}^{1}\frac{1}{s}yg(y)dy + G(s)\\
\end{split}\]

Where we used the fact that $\int_{0}^{1}(1/s)yg(y)dy = \mu - c$ and $s = n\mu(1-a)^{n-1}$.
\end{proof}

\subsection{Equation \texorpdfstring{\ref{l}}{TEXT} Derivation}\label{a0}

\begin{claim}
\[\mu \sum_{i=0}^{n-1}\binom{n-1}{i}\bigg(\frac{1}{n-i}\bigg)(1-a)^{i}a^{n-1-i} = \mu\frac{1 - (1-a)^{n}}{na}\]
\end{claim}

\begin{proof}
First, define $k \coloneqq n -1 -i$ and so we have 

\[\mu \sum_{i=0}^{n-1}\binom{n-1}{i}\bigg(\frac{1}{n-i}\bigg)(1-a)^{i}a^{n-1-i} = \mu (1-a)^{n-1}\sum_{k=0}^{n-1}\binom{n-1}{k}\frac{1}{k+1}\bigg(\frac{a}{1-a}\bigg)^{k}\]
Then, we have the identity. %(see \cite{ms1}).

\[\sum_{k=0}^n\frac{1}{k+1}\binom{n}{k}z^k=\frac{(z+1)^{n+1}-1}{(n+1)z}\]

and so we simply set $z \coloneqq \frac{a}{1-a}$, and after some algebra the proof is completed.
\end{proof}

\subsection{Lemma \ref{lem7} Proof}\label{a1}

\begin{proof}

Let Players $2$ through $n$ play $F_{i}$ supported on $[t,s]$ and have a point mass of size $a$ on $1$. Suppose for the sake of contradiction that $t > 0$. Recall,

\[F_{i}(x) = (1-a)\bigg(\frac{x-t}{s-t}\bigg)^{1/(n-1)}\]

Thus, the cdf of the maximum of this, $H \coloneqq F^{n-1}$, is

\[H(x) = (1-a)^{n-1}\bigg(\frac{x-t}{s-t}\bigg)\]

and

\[h(x) = (1-a)^{n-1}\bigg(\frac{1}{s-t}\bigg)\]

Let Player $1$ play some alternate strategy $G$ supported on $[0,s]$ such that the density is positive on some portion of $[0,t]$ and have a point mass of size $a$ on $1$. Then, player $1$'s expected payoff is:

\[\begin{split}
    u &= \frac{1 - (1-a)^{n}}{n} + \int_{t}^{s}\int_{t}^{y}h(x)g(y)dxdy\\
    &= \frac{1 - (1-a)^{n}}{n} + \int_{t}^{s}(1-a)^{n-1}\bigg(\frac{y}{s-t}\bigg)g(y)dy - \int_{t}^{s}(1-a)^{n-1}\bigg(\frac{t}{s-t}\bigg)g(y)dy\\
    &= \frac{1 - (1-a)^{n}}{n} + \int_{0}^{s}(1-a)^{n-1}\bigg(\frac{y}{s-t}\bigg)g(y)dy - \int_{0}^{t}(1-a)^{n-1}\bigg(\frac{y}{s-t}\bigg)g(y)dy\\ 
    &- \int_{0}^{s}(1-a)^{n-1}\bigg(\frac{t}{s-t}\bigg)g(y)dy + \int_{0}^{t}(1-a)^{n-1}\bigg(\frac{t}{s-t}\bigg)g(y)dy\\
    &= \frac{1 - (1-a)^{n}}{n} + \int_{0}^{s}(1-a)^{n-1}\bigg(\frac{y}{s-t}\bigg)g(y)dy - \int_{0}^{s}(1-a)^{n-1}\bigg(\frac{t}{s-t}\bigg)g(y)dy\\ 
    &- \int_{0}^{t}(1-a)^{n-1}\bigg(\frac{y}{s-t}\bigg)g(y)dy + \int_{0}^{t}(1-a)^{n-1}\bigg(\frac{t}{s-t}\bigg)g(y)dy\\
    &= \frac{1}{n} + \int_{0}^{t}(1-a)^{n-1}\bigg(\frac{t - y}{s-t}\bigg)g(y)dy > \frac{1}{n}
\end{split}\]

where we used,

\[\label{ha}\tag{$A1$}\begin{split}
    &\frac{1 - (1-a)^{n}}{n} + \int_{0}^{s}(1-a)^{n-1}\bigg(\frac{y}{s-t}\bigg)g(y)dy - \int_{0}^{s}(1-a)^{n-1}\bigg(\frac{t}{s-t}\bigg)g(y)dy\\
    &= \frac{1 - (1-a)^{n}}{n} + (1-a)^{n-1}\bigg(\frac{\mu - a}{s-t}\bigg) - (1-a)^{n}\bigg(\frac{t}{s-t}\bigg)\\
\end{split}\]

But, we can find $\mu - a$ explicitly. It is,

\[\begin{split}
    \mu - a &= \int_{t}^{s}yf(y)dy\\
            &= \int_{t}^{s}\frac{y}{y-t}\bigg(\frac{1-a}{n-1}\bigg)\bigg(\frac{y-t}{s-t}\bigg)^{1/(n-1)}dy\\
            &= \left.\dfrac{\left(1-a\right)\left(y-t\right)^\frac{1}{n-1}\left(y+\left(n-1\right)t\right)}{n\left(s-t\right)^\frac{1}{n-1}}\right\vert_{t}^{s}\\
            &= \dfrac{\left(1-a\right)\left(s+\left(n-1\right)t\right)}{n}
\end{split}\]

Substitute this into \ref{ha}:

\[\begin{split}
    &\frac{1 - (1-a)^{n}}{n} + (1-a)^{n-1}\bigg(\frac{\mu - a}{s-t}\bigg) - (1-a)^{n}\bigg(\frac{t}{s-t}\bigg)\\
    &= \frac{1 - (1-a)^{n}}{n} + (1-a)^{n}\bigg(\dfrac{\left(s+\left(n-1\right)t\right)}{n(s-t)}\bigg) - (1-a)^{n}\bigg(\frac{t}{s-t}\bigg)\\
    &= \frac{1 - (1-a)^{n}}{n} + (1-a)^{n}\bigg(\dfrac{\left(s+\left(n-1\right)t - nt\right)}{n(s-t)}\bigg)\\
    &= \frac{1 - (1-a)^{n}}{n} + (1-a)^{n}\bigg(\dfrac{s-t}{n(s-t)}\bigg)\\
    &= \frac{1}{n}
\end{split}\]

Thus, there is a profitable deviation and so $t$ must equal $0$.
\end{proof}

\subsection{Lemma \ref{lem8} Proof}\label{a2}

\begin{proof}

Recall that in Lemma \ref{blah} we show that $a \geq \mu - \mu(1-a)^{n}$. Thus, it is sufficient to show here that $a \leq \mu - \mu(1-a)^{n}$.

We divide the following into two cases. In the first case, suppose that $\mu \leq s$. Note that it cannot be a profitable deviation for a player to play a strategy consisting of $s$ played with probability $\mu/s$ and $0$ with probability $1 - \mu/s$. This condition is equivalent to \ref{23}:

\[\tag{$A2$}\label{23}\begin{split}
    \frac{1}{n} &\geq \frac{(1-a)^{n-1}\mu}{s}\\
\end{split}\]

Or, 

\[\tag{$A3$}\label{237}\begin{split}
    s &\geq n (1-a)^{n-1}\mu\\
\end{split}\]

From \ref{2}, and using the fact that $t = 0$, we have

\[a= \frac{n\mu - s}{n - s}\]

Or,

\[\label{123}\tag{$A4$} s = \frac{n(\mu-a)}{(1-a)}\]

We substitute this into \ref{237} and obtain

\[\begin{split}
    \frac{n(\mu-a)}{(1-a)} &\geq n (1-a)^{n-1}\mu\\
    \mu - \mu(1-a)^{n} &\geq a
\end{split}\]

For the second case, suppose now that $\mu > s$. By similar logic to the above, it cannot be a profitable deviation for a player to play a strategy consisting of $s$ played with probability $(1-\mu)/(1-s)$ and $1$ with probability $(\mu-s)/(1-s)$. That is,

\[\label{8}\tag{$A5$}\frac{1}{n} \geq \bigg(\frac{1-\mu}{1-s}\bigg)(1-a)^{n-1} + \bigg(\frac{\mu - s}{1-s}\bigg)\bigg(\frac{1 - (1-a)^{n}}{na}\bigg)\]

Suppose for the sake of contradiction that $a > \mu - \mu(1-a)^{n}$. Additionally, for convenience, define $k \coloneqq (1-a)^{n}$. \ref{8} can be rearranged to obtain:

\[s\big(1-k-a\big) \geq a n (1-\mu)\frac{k}{(1-a)} + \mu(1-k) - a\]

We substitute \ref{123} and rearrange to obtain:

\[\label{77}\tag{$A6$}\begin{split}
    n(\mu-a)\big(1-k-a\big) &\geq ank(1-\mu) + \mu(1-k)(1-a) - a(1-a)\\
    \mu\big(n+a-an-1\big) - (1-a)a(n-1) &\geq \mu k \big((n-1)(1-a)\big)
\end{split}\]

Our assumption above that $a > \mu - \mu(1-a)^{n}$ is equivalent to $\mu k > \mu - a$. We substitute this into \ref{77} and cancel:

\[\begin{split}
    \mu\big(n+a-an-1\big) - (1-a)a(n-1) &\geq \mu k \big((n-1)(1-a)\big)\\
    \mu\big(n+a-an-1\big) - (1-a)a(n-1) &> (\mu - a)\big((n-1)(1-a)\big)\\
    0 &> 0
\end{split}\]

We have achieved a contradiction and have thereby shown that $a \geq \mu - \mu(1-a)^{n}$. This, combined with Lemma \ref{blah} allows us to conclude the result, that $a = \mu - \mu(1-a)^{n}$.

\end{proof}

\end{document}